\documentclass[10pt]{amsart}

\usepackage[latin1]{inputenc}
\usepackage[english]{babel}
\usepackage{amsmath}
\usepackage{amsfonts}
\usepackage{amssymb}
\usepackage{amsthm}
\usepackage{epsfig}
\usepackage{graphicx,colortbl}
\usepackage{url,hyperref}
\usepackage{mathtools, color, changes}

\newtheorem{theorem}{Theorem}[section]
\newtheorem{lemma}[theorem]{Lemma}
\newtheorem{corollary}[theorem]{Corollary}
\newtheorem{proposition}[theorem]{Proposition}

\newtheorem{remark}[theorem]{Remark}

\newcounter{theor}

\newtheorem{thm}[theor]{Theorem}

\def\lin{\mathop\mathrm{lin}\nolimits}
\def\conv{\mathop\mathrm{conv}\nolimits}

\def\hyp{\mathop\mathrm{hyp}\nolimits}

\def\K{\mathcal{K}}

\def\L{\mathcal{L}}

\def\R{\mathbb{R}}
\def\N{\mathbb{N}}
\def\Z{\mathbb{Z}}

\def\vol{\mathrm{vol}}

\def\e{\mathrm{e}}
\newcommand{\dlat}{\mathrm{d}}

\def\e{\mathrm{e}}
\def\esc#1{\left\langle #1\right\rangle}

\def\Cube#1{\mathrm{C}_{#1}}
\def\d{\mathcal{D}}

\def\floor#1{\lfloor#1\rfloor}
\def\2Z{2^{-m}\Z^n}
\def\Gsub#1{\mathrm{G}_{#1}}
\def\G{\mathrm{G}_n}
\def\symbol{\diamond}

\DeclareMathOperator*{\esssup}{ess\,sup}

\numberwithin{equation}{section}

\begin{document}

\title[Rogers-Shephard inequalities for the lattice point enumerator]
{On Rogers-Shephard type inequalities for the lattice point enumerator}

\author{David Alonso-Guti\'errez}
\address{\'Area de an\'alisis matem\'atico, Departamento de matem\'aticas, Facultad de Ciencias, Universidad de Zaragoza, Pedro cerbuna 12, 50009 Zaragoza (Spain), IUMA}
\email{alonsod@unizar.es}
\author{Eduardo Lucas}
\author{Jes\'us Yepes Nicol\'as}
\address{Departamento de Matem\'aticas, Universidad de Murcia, Campus de
Espinar\-do, 30100-Murcia, Spain}
\email{eduardo.lucas@um.es}
\email{jesus.yepes@um.es}

\thanks{First author is supported by MICINN project PID-105979-GB-I00 and DGA project E48\_20R.
Second and third authors are supported by MICINN/FEDER project PGC2018-097046-B-I00
and by ``Programa de Ayudas a Grupos de Excelencia de la Regi\'on de Murcia'', Fundaci\'on S\'eneca,
19901/GERM/15.}



\begin{abstract}
In this paper we study various Rogers-Shephard type inequalities for
the lattice point enumerator $\Gsub{n}(\cdot)$ on $\R^n$.
In particular,
for any non-empty convex bounded sets $K,L\subset\R^n$,
we show that
\[
\Gsub{n}(K+L)\Gsub{n}\bigl(K\cap(-L)\bigr)
\leq\binom{2n}{n} \Gsub{n}\bigl(K+(-1,1)^n\bigr)\Gsub{n}\bigl(L+(-1,1)^n\bigr).
\]
and
\[
\Gsub{n-k}(P_{H^\perp} K)\Gsub{k}(K\cap H)\leq\binom{n}{k}\Gsub{n}\bigl(K+(-1,1)^n\bigr),
\]
for $H=\lin\{\e_1,\dots,\e_k\}$, $k\in\{1,\dots,n-1\}$.

\smallskip

Additionally, a discrete counterpart to a classical result by Berwald for concave functions, from which other discrete Rogers-Shephard type inequalities may be derived, is shown.
Furthermore, we prove that these new discrete analogues for $\Gsub{n}(\cdot)$ imply the corresponding results involving the Lebesgue measure.
\end{abstract}

\maketitle

\section{Introduction and main results}

Let $\K^n$ be the set of all convex bodies of the Euclidean space
$\R^n$, i.e., the family of all non-empty compact convex sets in $\R^n$.
The $n$-dimensional volume of a measurable set $M\subset\R^n$, i.e., its $n$-dimensional
Lebesgue measure, is denoted by $\vol(M)$ or $\vol_n(M)$ if the distinction of the dimension
is useful (when integrating, as usual, $\dlat x$ will stand for $\dlat \vol(x)$).
The \emph{Minkowski sum} of two non-empty sets $A,B\subset\R^n$ is the
classical vector addition of them, $A+B=\{a+b:\, a\in A, \, b\in B\}$, and
we write $A-B$ for $A+(-B)$.
We denote by $A\sim B =\{x\in\R^n:\, x+B\subset A\}$ the so-called \emph{Minkowski difference} of $A$ and $B$
(for more on this notion and its connection with the Minkowski sum, we refer the reader to \cite[Section~3.1]{Sch2}).
Moreover, $\lambda A$ represents the set
$\{\lambda a:\, a\in A\}$ for $\lambda\geq0$, and $\dim A$ denotes its dimension, i.e., the dimension of its affine hull.

A fundamental relation involving the volume and the Minkowski addition is
the \emph{Brunn-Minkowski inequality}. One form of it states that if
$K,L\in\K^n$ and $\lambda\in(0,1)$ then
\begin{equation}\label{e:BM}
\vol\bigl((1-\lambda) K+\lambda L\bigr)^{1/n}\geq
(1-\lambda)\vol(K)^{1/n}+\lambda\vol(L)^{1/n}.
\end{equation}
The Brunn-Minkowski inequality has become not only a cornerstone of the
Brunn-Minkowski theory (for which we refer the reader to the updated
monograph \cite{Sch2}) but also a powerful tool in other related fields of
mathematics.
Moreover, it quickly yields other well-known inequalities,
such as the \emph{isoperimetric inequality}, it
has inspired new engaging related results and it has been the starting
point for new extensions and generalizations (see~e.g.~\cite[Chapter~9]{Sch2}).
For extensive survey articles on this and other
related inequalities we refer the reader to \cite{Brt,G}.

In the particular case when $L=-K$ and $\lambda=1/2$, \eqref{e:BM} gives
\[\vol(K-K)\geq 2^n\vol(K).\]
An upper bound for the volume of the difference body $K-K$ is given by the
\emph{Rogers-Shephard inequality}, originally proven in \cite{RS1}. For more details
about this inequality, we also refer the reader to
\cite[Section~10.1]{Sch2}.

\begin{thm}[The Rogers-Shephard inequality]\label{t:RS_vol}
Let $K\in\K^n$. Then
\begin{equation}\label{e:RS_K-K_vol}
\vol(K-K)\leq \binom{2n}{n}\vol(K).
\end{equation}
\end{thm}
This relation for $K-K$ can be generalized to the Minkowski addition of two convex bodies
$K,L\in\K^n$ as follows:
\begin{equation}\label{e:RS_K_L_vol}
\vol(K+L)\vol\bigl(K\cap(-L)\bigr)\leq\binom{2n}{n}\vol(K)\vol(L).
\end{equation}

The Rogers-Shephard inequality was recently
extended to the functional setting \cite{AAGJV,AGMJV,AEFO,Co},
generalized to different types of measures \cite{AHCYNRZ,Ro},
as well as studied in the $L_p$ setting \cite{A,BC}.
Moreover, it was recently extended to other geometric functionals \cite{AHCYN},
and a reverse form of Rogers-Shephard inequality
in the setting of log-concave functions was given in \cite{A}.
The role of this inequality in characterization results of the difference body was also
studied in \cite{ASG}, and it was
proven an optimal stability version of it in \cite{Bo}.
It is also interesting to note that a strengthening of this inequality
for mixed volumes was conjectured (independently by Godbersen and Makai~Jr.,
see~\cite[Note~5 for Section~10.1]{Sch2} and the references therein);
a conjecture on which engaging progress was recently made in \cite{AEFO}.

\smallskip

In \cite{GG} Gardner and Gronchi obtained a powerful discrete analogue of
the following form of the Brunn-Minkowski inequality, in the setting of
$\Z^n$ with the cardinality $|\cdot|$: $\vol(K+L)\geq\vol(B_K+B_L)$, where
$B_K$ and $B_L$ denote centered Euclidean balls of the same volume as $K$
and $L$, respectively. Moreover, from the mentioned version, they derive
some inequalities that improve previous results by Ruzsa, collected in
\cite{Ru1,Ru2}.

More recently, different discrete analogues of the Brunn-Minkowski
inequality have been obtained, including the case of its classical form
(cf.~\eqref{e:BM}) for the cardinality \cite{GT,HCIYN,IYNZ}, functional
extensions of it \cite{HKS,IYN,IYNZ,KL,Sl} and versions for the \emph{lattice
point enumerator} $\Gsub{n}(\cdot)$ \cite{HKS,HCLYN,IYNZ}, which is given by
$\Gsub{n}(M)=|M\cap\Z^n|$. In this respect, \cite[Theorem~2.1]{IYNZ} reads as
follows:
\begin{thm}\label{t: BM_lattice_point_no_G(K)G(L)>0}
Let $\lambda\in(0,1)$ and let $K,L\subset\R^n$ be non-empty bounded sets.
Then
\begin{equation}\label{e: BM_lattice_point_no_G(K)G(L)>0}
\G\bigl((1-\lambda)K+\lambda L+(-1,1)^n\bigr)^{1/n}\geq(1-\lambda)\G(K)^{1/n}+\lambda\G(L)^{1/n}.
\end{equation}
\end{thm}

Apart from the above-mentioned discrete analogues of the Brunn-Minkowski inequality,
various discrete counterparts, for the lattice point enumerator $\Gsub{n}(\cdot)$, of results in Convex Geometry were recently proven. Some examples of such results are \emph{Koldobsky's slicing inequality} \cite{AHZ},
\emph{Meyer's inequality} \cite{FH} and an isoperimetric type inequality \cite{ILYN}. We refer the reader to these articles and the references therein for other connected problems, questions and results.

At this point, and taking into account the strong connection between the Brunn-Minkowski inequality and the Rogers-Shephard inequality \eqref{e:RS_K-K_vol}, it is natural to wonder about the possibility of obtaining a discrete version of the latter, when dealing with the lattice point enumerator.
To this aim, a very elegant discrete analogue of the Rogers-Shephard inequality
\eqref{e:RS_K-K_vol} in the planar case (in fact, a stronger version of it) was shown in \cite{GG}, as a consequence of \emph{Pick's theorem} jointly with \eqref{e:RS_K-K_vol}:
\begin{thm}\label{t:GG}
Let $P\subset\R^2$ be a convex polygon with integer vertices. Then
\begin{equation}\label{e:GG}
\Gsub{2}(P-P)\leq 6\Gsub{2}(P)-\mathrm{b}(P)-5,
\end{equation}
where $\mathrm{b}(P)$ denotes the number of integer points in the boundary of $P$.
\end{thm}
However, when dealing with an arbitrary convex body $K\in\K^n$, one cannot expect to get a discrete counterpart of \eqref{e:RS_K-K_vol} for the lattice point enumerator $\Gsub{n}(\cdot)$, namely,
\begin{equation*}
\Gsub{n}(K-K)\leq\binom{2n}{n}\Gsub{n}(K).
\end{equation*}
Indeed, just considering $K=[-1/2,1/2]^n$ one would obtain $3^n\leq \binom{2n}{n}$, which
is false for $n=1,2,3,4$. Moreover, as pointed out in \cite{FH}, where the authors consider certain
simplices with integer vertices, there is neither a possible extension of \eqref{e:GG} in dimension $n\geq3$
nor even a hope to get $\Gsub{n}(K-K)\leq c_n\Gsub{n}(K)$ for some constant $c_n>0$ depending only on
the dimension $n$, for $n\geq3$.

Altogether, and taking into account the ``behavior'' of the discrete version of the Brunn-Minkowski inequality
collected in \eqref{e: BM_lattice_point_no_G(K)G(L)>0},
an alternative to get such an inequality for the lattice
point enumerator would be to consider some extension of $K$ (by Minkowski adding certain cube) on the right-hand
side of \eqref{e:RS_K-K_vol} (or, more generally, on that of \eqref{e:RS_K_L_vol}).
In this regard, here we show the following:

\begin{theorem}\label{t:RS_K-K}
Let $K\subset\R^n$ be a non-empty convex bounded set. Then
\begin{equation}\label{e:RS_K-K}
\Gsub{n}(K-K)\leq\binom{2n}{n}\Gsub{n}\left(K+\left(-\frac{3}{4},\frac{3}{4}\right)^n\right).
\end{equation}
\end{theorem}

When considering the Minkowski sum of two non-empty convex bounded sets $K,L\subset\R^n$, instead of $K-K$, we have the following discrete counterpart of \eqref{e:RS_K_L_vol}:
\begin{theorem}\label{t:RS_K_L}
Let $K,L\subset\R^n$ be non-empty convex bounded sets and
let
\[
c_{_{K,L}}=\frac{\vol\left(K+L+\left(-\frac{1}{2},\frac{1}{2}\right)^n\right)}{\vol\bigl(K+L+(-1,1)^n\bigr)}\in(0,1).
\]
Then
\begin{equation}\label{e:RS_K_L}
\Gsub{n}(K+L)\Gsub{n}\bigl(K\cap(-L)\bigr)
\leq\binom{2n}{n} c_{_{K,L}}\,\Gsub{n}\bigl(K+(-1,1)^n\bigr)\Gsub{n}\bigl(L+(-1,1)^n\bigr).
\end{equation}
In particular, taking $L=-K$, with $0\in K$,
\begin{equation}\label{e:RS_K-K_2}
\Gsub{n}(K-K)
\leq\binom{2n}{n} c_{_{K,-K}}\,\frac{\Gsub{n}\bigl(K+(-1,1)^n\bigr)^2}{\Gsub{n}(K)}.
\end{equation}
\end{theorem}
Both \eqref{e:RS_K-K} and \eqref{e:RS_K_L} (and thus also \eqref{e:RS_K-K_2}) are \emph{asymptotically
sharp} (in the sense that, for any of these inequalities, there exist convex bodies containing the origin in such a way that the ratio between the right-hand and the left-hand sides, when applied to dilations of these sets scaled by a factor $r>0$, tends to $1$ as $r\to\infty$;
see Remark~\ref{r:sharp}) and, even more, they imply the continuous versions \eqref{e:RS_K-K_vol} and \eqref{e:RS_K_L_vol}, respectively.

Furthermore, we will also show in Section~\ref{s:RS} an alternative discrete version of \eqref{e:RS_K_L_vol}
involving both the Minkowski difference and the addition of the sets and the cube $(-1,1)^n$ (see~Theorem \ref{t:RS_diff_refined}).

\medskip

In \cite[Theorem~1]{RS2}, Rogers and Shephard also gave the following
lower bound for the volume of a convex body $K\in\K^n$ in terms of the volumes of a projection
and a section of $K$. Before recalling its precise statement we need some auxiliary notation:
the set of all $k$-dimensional linear subspaces of $\R^n$ is denoted by $\L^n_{k}$,
and for $H\in\L^n_{k}$, the orthogonal projection of $M$ onto $H$ is denoted by $P_HM$;
moreover, as usual, $H^{\perp}\in \L^n_{n-k}$ represents the orthogonal complement of $H$.
\begin{thm}\label{t:RS_sect_proj_vol}
Let $k\in\{1,\dots,n-1\}$ and $H\in\L^n_{k}$. Let $K\in\K^n$ be a convex body. Then
\begin{equation}\label{e:RS_sect_proj_vol}
\vol_{n-k}(P_{H^\perp} K)\vol_k(K\cap H)\leq\binom{n}{k}\vol(K).
\end{equation}
\end{thm}
In this paper we will show that the above result also admits a suitable discrete version for
the lattice point enumerator $\Gsub{n}(\cdot)$. More precisely, we obtain:
\begin{theorem}\label{t:RS_sect_proj}
Let $k\in\{1,\dots,n-1\}$ and $H=\lin\{\e_1,\dots,\e_k\}\in\L^n_{k}$.
Let $K\subset\R^n$ be a non-empty convex bounded set. Then
\begin{equation}\label{e:RS_sect_proj}
\Gsub{n-k}(P_{H^\perp} K)\Gsub{k}(K\cap H)\leq\binom{n}{k}\Gsub{n}\bigl(K+(-1,1)^n\bigr).
\end{equation}
\end{theorem}

\smallskip

A classical result due to Berwald \cite{Ber} (see also \cite{AAGJV,ABG} for other extensions and considerations),
from which the Rogers-Shephard inequalities \eqref{e:RS_K-K_vol} and \eqref{e:RS_sect_proj_vol} can be derived,
relates certain weighted power means of a concave function, as follows:
\begin{thm}[Berwald's inequality]\label{t:Berwald_classical}
Let $K\in\K^n$ be a convex body with $\dim K=n$ and let $f:K\longrightarrow\R_{\geq0}$ be a concave function.
Then, for any $0<p<q$,
\begin{equation}\label{e:Berwald_classical}
\left(\frac{\binom{n+q}{n}}{\vol(K)}\int_{K}f^{q}(x)\,\dlat x\right)^{1/q}
\leq\left(\frac{\binom{n+p}{n}}{\vol(K)}\int_{K}f^{p}(x)\,\dlat x\right)^{1/p}.
\end{equation}
\end{thm}

As mentioned, it is remarkable to point out that the above result provides one with an alternative unified proof of both Theorems \ref{t:RS_vol} and \ref{t:RS_sect_proj_vol}, either by considering the function $f:P_{H^\perp}K\longrightarrow\R_{\geq0}$ given by \[f(x)=\vol_k\bigl(K\cap(x+H\bigr)\bigr)^{1/k}\]
for $H\in\L^n_k$, $p=k$ (and $n'=n-k$), and letting $q\to\infty$, or by considering the function $f:K-K\longrightarrow\R_{\geq0}$ defined by \[f(x)=\vol\bigl(K\cap(x+K)\bigr)^{1/n},\]
$p=n$, and letting $q\to\infty$ (see~Section
\ref{s:Berwald} for the precise details on the  obtention of Theorems \ref{t:RS_vol} and \ref{t:RS_sect_proj_vol} from Theorem \ref{t:Berwald_classical}).


Therefore, one may wonder whether Berwald's inequality also admits a discrete analogue. To establish its statement, we need the following notation (see \cite{IYNZ} and the references therein for more about this notion): for a function $\phi:K\longrightarrow\R_{\geq0}$ defined on a convex bounded set $K\subset\R^n$, we denote by $\phi^{\symbol}:K+(-1,1)^n\longrightarrow\R_{\geq0}$ the function given by	
\[\phi^{\symbol}(z) = \sup_{u\in(-1,1)^n}\overline{\phi}(z+u)  \quad \text{ for all } z\in K+(-1,1)^n,\]
where $\overline{\phi}:\R^n\longrightarrow\R_{\geq0}$ is the function defined by
\[
\overline{\phi}(x)=\left\{
\begin{array}{ll}
\phi(x) & \text{ if } x\in K,\\[2mm]
0 & \text{ otherwise}.
\end{array}
\right.
\]
In other words, $\phi^{\symbol}$ is the function whose hypograph is the closure of the Minkowski sum of the hypograph of $\phi$ and $(-1,1)^n\times\{0\}$.

Here we show the following discrete analogue of Berwald's inequality \eqref{e:Berwald_classical}:
\begin{theorem}\label{t:Berwald}
Let $K\subset\R^n$ be a convex bounded set containing the origin
and let $f:K\longrightarrow\R_{\geq0}$ be a concave function with $f(0)=|f|_\infty$.
Then, for any $0<p<q$,
\begin{equation}\label{e:Berwald}
\left(\frac{\binom{n+q}{n}}{\Gsub{n}(K)}\sum_{x\in K\cap\Z^n}f^{q}(x)\right)^{1/q}
\leq\left(\frac{\binom{n+p}{n}}{\Gsub{n}(K)}\sum_{x\in (K+(-1,1)^n)\cap\Z^n}
\bigl(f^{\symbol}\bigr)^{p}(x)\right)^{1/p}.
\end{equation}
\end{theorem}

\begin{remark}
Let us point out that the origin does not play any role in the latter theorem and the assumption on the origin could be substituted by $|f|_\infty=f(x_0)$ for some $x_0\in\Z^n$. It is, however, needed in our proof that, that the maximum of $f$ is attained at some point in $\Z^n$.
\end{remark}
As we will see along the manuscript, all our discrete analogues for the lattice point enumerator
$\Gsub{n}(\cdot)$ imply the corresponding classical inequalities (see Theorems \ref{t:RS_disc_imply_cont} and \ref{t:Berw_disc_imply_cont}).

\smallskip

The paper is organized as follows:
in Section \ref{s:RS-projsect} we recall some preliminaries and we derive Theorem \ref{t:RS_sect_proj}
as a consequence of a stronger inequality, collected in Theorem \ref{t:RS_sect_proj_refined}.
In Section \ref{s:RS} we obtain our discrete Rogers-Shephard type inequalities
(in particular, we prove both Theorems \ref{t:RS_K-K} and \ref{t:RS_K_L}) and we show that these discrete
analogues imply the corresponding continuous versions.
Finally, Section \ref{s:Berwald} is mainly devoted to the proof of Theorem \ref{t:Berwald}.

\section{Discrete projection-section inequalities}\label{s:RS-projsect}

We start this section by recalling some notions and fixing some notation that will be needed throughout the rest of the manuscript. We shall work in the $n$-dimensional Euclidean space $\R^n$, with origin $0$ (or $0_n$ if the distinction of the dimension is useful), endowed with the standard inner product
$\esc{\cdot,\cdot}$, and we will write $\e_i$ to represent the $i$-th canonical unit vector.
Given a non-empty set $M\subset\R^n$, let $\conv M$ and $\lin M$ denote, respectively, the convex and linear hulls of $M$, i.e., the smallest convex set and vector subspace, respectively, containing the set $M$.
Moreover, $\chi_{_M}$ will represent the characteristic function of $M$.
We will also write
\[\hyp(f)=\bigl\{(x,t): \, x\in M, \, t\in\R, \, f(x)\geq t \bigr\}\subset\R^{n+1}\]
for the \emph{hypograph} of a non-negative function $f:M\longrightarrow\R_{\geq0}$.

For the vector subspace $H=\lin\{\e_1,\dots,\e_k\}\in\L^n_{k}$, $k\in\{1,\dots,n-1\}$ we denote by
\[\Gsub{k}(M)=\bigl|M\cap\bigl(x+(\Z^{k}\times\{0_{n-k}\})\bigr)\bigr|\]
for any $M\subset x+H$ for some $x\in H^\perp$. Analogously,
we write \[\Gsub{n-k}(M)=\bigl|M\cap\bigl(y+(\{0_{k}\}\times\Z^{n-k})\bigr)\bigr|\]
for each $M\subset y+H^\perp$, for $y\in H$.
Furthermore, for the sake of simplicity, we will write $\Cube{H}:=(-1,1)^n\cap H$.

Finally, as usual in the literature, we will use the following conventional notation:
\[\binom{r}{s}:=\frac{\Gamma(r+1)}{\Gamma(s+1)\Gamma(r-s+1)}\]
for any $r,s>0$, where $\Gamma(\cdot)$ denotes the \emph{Gamma function}.

\medskip

Now we prove Theorem \ref{t:RS_sect_proj}, under the mild assumption that $K$ contains the origin, by showing the following more general result. In Remark~\ref{r:new proof Theor projsect} we will give an alternative proof of Theorem~ \ref{t:RS_sect_proj}, also valid for the case of any non-empty convex bounded set $K$ (not necessarily containing the origin).

\begin{theorem}\label{t:RS_sect_proj_refined}
Let $k\in\{1,\dots,n-1\}$ and $H=\lin\{\e_1,\dots,\e_k\}\in\L^n_{k}$.
Let $K\subset\R^n$ be a convex bounded set containing the origin. Then
\begin{equation}\label{e:RS_sect_proj_refined}
\left[\sum_{i=0}^{n-k}\frac{k}{n-i}\binom{n-k}{i}\binom{n}{i}^{-1}
\Gsub{n-k}(P_{H^\perp} K)^{i/(n-k)}\right]\Gsub{k}(K\cap H)
\leq\Gsub{n}\bigl(K+(-1,1)^n\bigr).
\end{equation}
\end{theorem}

\smallskip

In particular, taking only the terms corresponding to $i=0$ and $i=n-k$, we obtain
\begin{equation*}\label{e:RS_sect_proj_refined2}
\left[\binom{n-1}{k-1}+\Gsub{n-k}(P_{H^\perp} K)\right]\Gsub{k}(K\cap H)
\leq\binom{n}{k}\Gsub{n}\bigl(K+(-1,1)^n\bigr).
\end{equation*}

\begin{proof}
First, for any $r\geq0$, we define the superlevel set
\begin{equation*}
\begin{split}
\d_r&=\Bigl\{x\in P_{H^\perp}\bigl(K+(-1,1)^n\bigr): \, \Gsub{k}\Bigl(\bigl(K+(-1,1)^n\bigr)\cap(x+H)\Bigr)\geq r\Bigr\}\\
&=\Bigl\{x\in P_{H^\perp}K+\Cube{H^\perp}: \, \Gsub{k}\Bigl(\bigl((K+\Cube{H^\perp})\cap(x+H)\bigr)+\Cube{H}\Bigr)\geq r\Bigr\}.
\end{split}
\end{equation*}
Now, let $x\in \d_0$, $y\in\Cube{H^\perp}$ and $\lambda\in[0,1]$. So, from the convexity of $K$, we have
\begin{equation}\label{e:ineq_Gk_in_proof_of_RS_sect_proj}
\begin{split}
\Gsub{k}&\Bigr(\bigl((K+\Cube{H^\perp})\cap\bigl((1-\lambda) x+\lambda y+H\bigr)\bigr)+\Cube{H}\Bigr)^{1/k}\\
&\geq \Gsub{k}\Bigl((1-\lambda)\bigl((K+\Cube{H^\perp})\cap (x+H)\bigr)
+\lambda\bigl((K+\Cube{H^\perp})\cap(y+ H)\bigr)+\Cube{H}\Bigr)^{1/k}.
\end{split}
\end{equation}
Notice that, since $x\in\d_0=P_{H^\perp} K+\Cube{H^\perp}$ and $y\in\Cube{H^\perp}\subset
P_{H^\perp} K+\Cube{H^\perp}$ (because $0\in K$), the sets
$(K+\Cube{H^\perp})\cap (x+H), (K+\Cube{H^\perp})\cap (y+H)$ are non-empty
and then the above sum
\[(1-\lambda)\bigl((K+\Cube{H^\perp})\cap (x+H)\bigr)
+\lambda\bigl((K+\Cube{H^\perp})\cap(y+ H)\bigr)+\Cube{H}\]
is well-defined.
Hence, from \eqref{e: BM_lattice_point_no_G(K)G(L)>0} we get
\begin{equation}\label{e:BM_in_proof_of_RS_sect_proj}
\begin{split}
\Gsub{k}&\Bigl((1-\lambda)\bigl((K+\Cube{H^\perp})\cap (x+H)\bigr)
+\lambda\bigl((K+\Cube{H^\perp})\cap(y+ H)\bigr)+\Cube{H}\Bigr)^{1/k}\\
&\geq (1-\lambda)\Gsub{k}\bigl((K+\Cube{H^\perp})\cap(x+H)\bigr)^{1/k}
+\lambda\Gsub{k}\bigl((K+\Cube{H^\perp})\cap(y+H)\bigr)^{1/k}\\
&\geq \lambda\Gsub{k}(K\cap H)^{1/k},
\end{split}
\end{equation}
where in the last inequality we have used that
\[\Gsub{k}\bigl((K+\Cube{H^\perp})\cap(y+H)\bigr)
\geq \Gsub{k}\bigl((y+K)\cap(y+H)\bigr)
=\Gsub{k}(K\cap H)\]
for every $y\in\Cube{H^\perp}$. Thus, setting
\begin{equation*}
\lambda_s=\left(\frac{s}{\Gsub{k}(K\cap H)}\right)^{1/k}
\end{equation*}
for any $0\leq s\leq \Gsub{k}(K\cap H)$ (observe that $\Gsub{k}(K\cap H)\neq0$ since $0\in K$),
from \eqref{e:ineq_Gk_in_proof_of_RS_sect_proj} and \eqref{e:BM_in_proof_of_RS_sect_proj}
for $\lambda=\lambda_s$ we conclude that
\[
(1-\lambda_s)\bigl(P_{H^\perp} K+\Cube{H^\perp}\bigr)
+ \lambda_s \Cube{H^\perp}\subset \d_s.
\]
In other words, for any $0\leq s\leq \Gsub{k}(K\cap H)$ we have
\[
(1-\lambda_s)P_{H^\perp} K+\Cube{H^\perp}\subset \d_s
\]
and then, by \eqref{e: BM_lattice_point_no_G(K)G(L)>0}, we get
\[
(1-\lambda_s)\Gsub{n-k}(P_{H^\perp} K)^{1/(n-k)}
+\lambda_s\Gsub{n-k}(\{0_{n-k}\})^{1/(n-k)}\leq \Gsub{n-k}(\d_s)^{1/(n-k)}
\]
and consequently
\begin{equation}\label{e:binom_lambda_s}
\sum_{i=0}^{n-k}\binom{n-k}{i}(1-\lambda_s)^{i}\lambda_s^{n-k-i}\Gsub{n-k}(P_{H^\perp} K)^{i/(n-k)}\leq \Gsub{n-k}(\d_s)
\end{equation}
for all $0\leq s\leq\Gsub{k}(K\cap H)$.

\smallskip

Now, on the one hand, we observe that
\begin{equation*}
\begin{split}
\int_0^{\Gsub{k}(K\cap H)}(1-\lambda_s)^{i}\lambda_s^{n-k-i}\,\dlat s
&=k\, \Gsub{k}(K\cap H)\int_0^1\theta^{n-i-1}(1-\theta)^{i}\,\dlat \theta\\
&=k\, \Gsub{k}(K\cap H)\frac{\Gamma(n-i)\Gamma(i+1)}{\Gamma(n+1)}\\
&=\frac{k}{n-i}\binom{n}{i}^{-1}\Gsub{k}(K\cap H)
\end{split}
\end{equation*}
and hence, integration on $s\in[0,\Gsub{k}(K\cap H)]$ on the left-hand side of \eqref{e:binom_lambda_s} yields
\[
\left[\sum_{i=0}^{n-k}\frac{k}{n-i}\binom{n-k}{i}\binom{n}{i}^{-1}
\Gsub{n-k}(P_{H^\perp} K)^{i/(n-k)}\right]\Gsub{k}(K\cap H).
\]
On the other hand, we have
\begin{equation*}
\begin{split}
&\int_0^{\Gsub{k}(K\cap H)} \Gsub{n-k}(\d_s)\,\dlat s\\
&\qquad=\int_0^{\Gsub{k}(K\cap H)} \sum_{x\in(P_{H^\perp}K+\Cube{H^\perp})\cap\Z^n}
\chi_{_{\d_s}}(x)\,\dlat s\\
&\qquad=\sum_{x\in(P_{H^\perp}K+\Cube{H^\perp})\cap\Z^n}
\min\left\{\Gsub{k}\Bigl(\bigl((K+\Cube{H^\perp})\cap(x+H)\bigr)+\Cube{H}\Bigr),
\Gsub{k}(K\cap H)\right\}\\
&\qquad\leq\sum_{x\in(P_{H^\perp}K+\Cube{H^\perp})\cap\Z^n}
\Gsub{k}\Bigl(\bigl((K+\Cube{H^\perp})\cap(x+H)\bigr)+\Cube{H}\Bigr)\\
&\qquad=\sum_{x\in (P_{H^\perp}K+\Cube{H^\perp})\cap\Z^n}\Gsub{k}\Bigl(\bigl(K+(-1,1)^n\bigr)\cap(x+H)\Bigr)\\
&\qquad=\Gsub{n}\bigl(K+(-1,1)^n\bigr).\\
\end{split}
\end{equation*}
This concludes the proof.
\end{proof}

\begin{remark}\label{r:H for proj-sect}
The role of $H=\lin\{\e_1,\dots,\e_k\}$ in the above result can be played by any other $k$-dimensional coordinate
(vector) subspace.
\end{remark}

Next we derive a first discrete analogue of the Rogers-Shephard inequality \eqref{e:RS_K_L_vol}, by considering a suitable $(2n)$-dimensional convex bounded set and applying Theorem \ref{t:RS_sect_proj}, following the original idea of Rogers and Shephard in \cite{RS2}:
\begin{corollary}\label{c:RS}
Let $K,L\subset\R^n$ be convex bounded sets containing the origin. Then
\begin{equation}\label{e:RS_K_L_weak}
\Gsub{n}(K+L)\Gsub{n}\bigl(K\cap(-L)\bigr)\leq \binom{2n}{n}\Gsub{n}\bigl(K+(-1,1)^n\bigr)
\Gsub{n}\bigl(L+(-2,2)^n\bigr).
\end{equation}
In particular, taking $L=-K$,
\begin{equation}\label{e:RS_K-K_weak}
\Gsub{n}(K-K)\leq\binom{2n}{n}\Gsub{n}\bigl(K+(-1,1)^n\bigr)
\frac{\Gsub{n}\bigl(K+(-2,2)^n\bigr)}{\Gsub{n}(K)}.
\end{equation}
\end{corollary}

\begin{proof}
Consider the $(2n)$-dimensional convex bounded set containing the origin defined by
\[
F=\bigl\{(x,y)\in\R^{2n}: \, x\in K, \, x-y\in-L\bigr\}
\]
and let $H=\lin\{\e_1,\dots,\e_n\}\in\L_n^{2n}$. Notice that $P_{H^\perp} F$ is the set of points $(0,y)$ such that $(x,y)\in F$ for some $x\in\R^n$, which is equivalent to the fact that $y\in x+L$ for some $x\in K$, and hence we get $P_{H^\perp} F=\{0_{n}\}\times (K+L)$. Moreover, we clearly have that \[F\cap H= \bigl(K\cap (-L)\bigr)\times\{0_{n}\}.\]
Now, given $(x,y)\in F+(-1,1)^{2n}$, we have that $x\in x_1+(-1,1)^n$ for some $x_1\in K$ and that $y\in y_1+(-1,1)^n$ for some $y_1\in x_1+L\subset x+(-1,1)^n+L$. So, for every $(x,y)\in F+(-1,1)^{2n}$, $x\in K+(-1,1)^n$ and $y\in x+L+(-2,2)^n$. Thus,
\begin{equation*}
\begin{split}
\Gsub{2n}\bigl(F+(-1,1)^{2n}\bigr)&\leq\sum_{x\in (K+(-1,1)^n)\cap\Z^n}\Gsub{n}\bigl(x+L+(-2,2)^n\bigr)\\
&=\Gsub{n}\bigl(K+(-1,1)^n\bigr)\Gsub{n}\bigl(L+(-2,2)^n\bigr).
\end{split}
\end{equation*}
Therefore, from Theorem \ref{t:RS_sect_proj} (applied to the convex bounded set $F$ containing the origin
and the vector subspace $H$) we obtain
\[
\Gsub{n}(K+L)\Gsub{n}\bigl(K\cap(-L)\bigr)\leq \binom{2n}{n}\Gsub{n}\bigl(K+(-1,1)^n\bigr)
\Gsub{n}\bigl(L+(-2,2)^n\bigr),
\]
which clearly further implies \eqref{e:RS_K-K_weak}. This concludes the proof.
\end{proof}

As we will see in the forthcoming section, we will derive other discrete Rogers-Shephard type inequalities, which are actually stronger than those (collected in Corollary \ref{c:RS}) that may be obtained as a consequence of applying the discrete projection-section inequality \eqref{e:RS_sect_proj}, in contrast to what happens in the continuous setting.

\section{Discrete Rogers-Shephard type inequalities}\label{s:RS}

When dealing with the (proof of the) Rogers-Shephard inequality \eqref{e:RS_K-K_vol}, one is naturally led to the notion of the \emph{covariogram} of a convex body $K\in\K^n$, that is, the function $f:\R^n\longrightarrow\R_{\geq0}$ given by
\[f(x)=\vol\bigl(K\cap(x+K)\bigr).\]
Its discrete version for finite sets $A\subset\R^n$, $\bigl|A\cap(x+A)\bigr|$, has been studied in \cite{GGZ}, where the authors show elegant relations of the latter with the continuous version. Here, first we will consider the following slight modification of the corresponding discrete version (for the lattice point enumerator) of the covariogram of $K$:
\[x\mapsto\Gsub{n}\Bigl(\bigl(K+(-1,1)^n\bigr)\cap\bigl(x+K+(-1,1)^n\bigr)\Bigr).\]
By using this, and exploiting the classical proof of the Rogers-Shephard inequality \eqref{e:RS_K-K_vol} that is based on the covariogram, we get to the following result. Again, we will present it in the more general setting of two convex bounded sets $K,L\subset\R^n$.

\begin{theorem}\label{t:RS_diff_refined}
Let $K,L\subset\R^n$ be convex bounded sets containing the origin, and such that $(-1,1)^n\subset L$.
Then
\begin{equation}\label{e:RS_diff_K_L_refined}
\begin{split}
\left[\sum_{i=0}^{n}\frac{n}{2n-i}\binom{n}{i}\binom{2n}{i}^{-1}
\Gsub{n}\bigl((K+L)\sim(-1,1)^n\bigr)^{i/n}\right]
\Gsub{n}\Bigl(K\cap\bigl((-L)\sim(-1,1)^n\bigr)\Bigr)\\
\leq\Gsub{n}\bigl(K+(-1,1)^n\bigr)\Gsub{n}\bigl(L+(-1,1)^n\bigr).
\end{split}
\end{equation}
In particular, taking $L=-K$ (for a convex set $K\subset\R^n$ with $(-1,1)^n\subset K$),
\begin{equation}\label{e:RS_diff_K-K_refined}
\sum_{i=0}^{n}\frac{n}{2n-i}\binom{n}{i}\binom{2n}{i}^{-1}
\Gsub{n}\bigl((K-K)\sim(-1,1)^n\bigr)^{i/n}
\leq\frac{\Gsub{n}\bigl(K+(-1,1)^n\bigr)^2}
{\Gsub{n}\bigl(K\sim(-1,1)^n\bigr)}.
\end{equation}
\end{theorem}

\smallskip

Before showing the result, we observe that taking only the terms corresponding to $i=0$ and $i=n$
in the above expressions we obtain
\begin{equation*}\label{e:RS_diff_K_L_refined2}
\begin{split}
\left[\frac{1}{2}\binom{2n}{n}+\Gsub{n}\bigl((K+L)\sim(-1,1)^n\bigr)\right]
\Gsub{n}\Bigl(K\cap\bigl((-L)\sim(-1,1)^n\bigr)\Bigr)\\
\leq\binom{2n}{n}\Gsub{n}\bigl(K+(-1,1)^n\bigr)
\Gsub{n}\bigl(L+(-1,1)^n\bigr)
\end{split}
\end{equation*}
and
\begin{equation*}\label{e:RS_diff_K-K_refined2}
\frac{1}{2}\binom{2n}{n}+\Gsub{n}\bigl((K-K)\sim(-1,1)^n\bigr)
\leq\binom{2n}{n}\frac{\Gsub{n}\bigl(K+(-1,1)^n\bigr)^2}
{\Gsub{n}\bigl(K\sim(-1,1)^n\bigr)},
\end{equation*}
respectively.

\smallskip

Moreover, taking only the term corresponding to $i=n$ applied to the sets $K$ and $L':=L+(-1,1)^n$ (and taking into account the relations between the Minkowski difference and addition; see~e.g.~\cite[Lemma~3.1.11]{Sch2}),
one recovers again the statement of Corollary \ref{c:RS}.


\begin{proof}
First, for any $r\geq0$, we consider the superlevel set
\[
\d_r=\Bigl\{x\in(K+L)+(-2,2)^n:\,\Gsub{n}\Bigl(\bigl(K+(-1,1)^n\bigr)\cap\bigl(x-L+(-1,1)^n\bigr)\Bigr)\geq r\Bigr\}.
\]
Now, let $x\in K+L$, $y\in(-1,1)^n$ and $\lambda\in[0,1]$. Then, from the convexity of $K$ and $L$, we have
\begin{equation}\label{e:ineq_Gn_in_proof_of_RS_covariogram}
\begin{split}
 \Gsub{n}&\Bigl(\bigl(K+(-1,1)^n\bigr)\cap\bigl((1-\lambda) x+\lambda y-L+(-1,1)^n\bigr)\Bigr)^{1/n}\\
&\geq \Gsub{n}\Bigl((1-\lambda)\bigl(K\cap(x-L)\bigr)+\lambda\bigl(K\cap(y-L)\bigr)+(-1,1)^n\Bigr)^{1/n}
\end{split}
\end{equation}
Notice that, since $x\in K+L$ and $y\in(-1,1)^n\subset L\subset K+L$ (because $0\in K$),
the sets $K\cap(x-L), K\cap(y-L)$ are non-empty
and then the above sum
\[(1-\lambda)\bigl(K\cap(x-L)\bigr)+\lambda\bigl(K\cap(y-L)\bigr)+(-1,1)^n\]
is well-defined.
Hence, from \eqref{e: BM_lattice_point_no_G(K)G(L)>0} we get
\begin{equation}\label{e:BM_in_proof_of_RS_covariogram}
\begin{split}
 \Gsub{n}&\Bigl((1-\lambda)\bigl(K\cap(x-L)\bigr)+\lambda\bigl(K\cap(y-L)\bigr)+(-1,1)^n\Bigr)^{1/n}\\
&\geq (1-\lambda) \Gsub{n}\bigl(K\cap(x-L)\bigr)^{1/n}+\lambda\Gsub{n}\bigl(K\cap(y-L)\bigr)^{1/n}\\
&\geq \lambda\Gsub{n}\Bigl(K\cap\bigl((-L)\sim(-1,1)^n\bigr)\Bigr)^{1/n},
\end{split}
\end{equation}
where in the last inequality we have used that
\[\Gsub{n}\bigl(K\cap(y-L)\bigr)
\geq \Gsub{n}\Bigl(K\cap\bigl((-L)\sim(-1,1)^n\bigr)\Bigr)\]
for every $y\in(-1,1)^n$.
Observing also that $\Gsub{n}\Bigl(K\cap\bigl((-L)\sim(-1,1)^n\bigr)\Bigr)\neq0$
since $0\in K\cap\bigl((-L)\sim(-1,1)^n\bigr)$, we may define
\[
\lambda_s=\left(\frac{s}{\Gsub{n}\Bigl(K\cap\bigl((-L)\sim(-1,1)^n\bigr)\Bigr)}\right)^{1/n}
\]
for any $0\leq s\leq \Gsub{n}\Bigl(K\cap\bigl((-L)\sim(-1,1)^n\bigr)\Bigr)$. Thus,
from \eqref{e:ineq_Gn_in_proof_of_RS_covariogram} and \eqref{e:BM_in_proof_of_RS_covariogram}
for $\lambda=\lambda_s$ we conclude that
\[
(1-\lambda_s)(K+L)
+ \lambda_s (-1,1)^n\subset \d_s.
\]
In particular, for any $0\leq s\leq \Gsub{n}\Bigl(K\cap\bigl((-L)\sim(-1,1)^n\bigr)\Bigr)$ we have

\[
\begin{split}
(1-\lambda_s)&\bigl((K+L)\sim(-1,1)^n\bigr)+(-1,1)^n \\
&= (1-\lambda_s)\Bigl(\bigl((K+L)\sim(-1,1)^n\bigr)+(-1,1)^n\Bigr)+\lambda_s(-1,1)^n\\
&\subset (1-\lambda_s)(K+L)+\lambda_s(-1,1)^n\subset\d_s.
\end{split}
\]
Then, by \eqref{e: BM_lattice_point_no_G(K)G(L)>0}, we get
\[
(1-\lambda_s)\Gsub{n}\bigl((K+L)\sim(-1,1)^n\bigr)^{1/n}
+\lambda_s\Gsub{n}(\{0\})^{1/n}\leq \Gsub{n}(\d_s)^{1/n}
\]
and consequently
\begin{equation}\label{e:binom_lambda_s_covariogram}
\sum_{i=0}^{n}\binom{n}{i}(1-\lambda_s)^{i}\lambda_s^{n-i}\Gsub{n}\bigl((K+L)\sim(-1,1)^n\bigr)^{i/n}
\leq \Gsub{n}(\d_s)
\end{equation}
for all $0\leq s\leq \Gsub{n}\Bigl(K\cap\bigl((-L)\sim(-1,1)^n\bigr)\Bigr)$.

\smallskip

Now, writing $s_0=\Gsub{n}\Bigl(K\cap\bigl((-L)\sim(-1,1)^n\bigr)\Bigr)$, on the one hand
we observe that
\begin{equation*}
\begin{split}
\int_0^{s_0}(1-\lambda_s)^{i}\lambda_s^{n-i}\,\dlat s
&=n\, \Gsub{n}\Bigl(K\cap\bigl((-L)\sim(-1,1)^n\bigr)\Bigr)\int_0^1\theta^{2n-i-1}(1-\theta)^{i}\,\dlat \theta\\
&=n\, \Gsub{n}\Bigl(K\cap\bigl((-L)\sim(-1,1)^n\bigr)\Bigr)\frac{\Gamma(2n-i)\Gamma(i+1)}{\Gamma(2n+1)}\\
&=\frac{n}{2n-i}\binom{2n}{i}^{-1}\Gsub{n}\Bigl(K\cap\bigl((-L)\sim(-1,1)^n\bigr)\Bigr)
\end{split}
\end{equation*}
and hence, integration on $s\in[0,s_0]$ on the left-hand side of \eqref{e:binom_lambda_s_covariogram} yields
\[
\left[\sum_{i=0}^{n}\frac{n}{2n-i}\binom{n}{i}\binom{2n}{i}^{-1}
\Gsub{n}\bigl((K+L)\sim(-1,1)^n\bigr)^{i/n}\right]\Gsub{n}\Bigl(K\cap\bigl((-L)\sim(-1,1)^n\bigr)\Bigr).
\]
On the other hand, we have
\begin{equation*}
\begin{split}
\int_0^{s_0} & \Gsub{n}(\d_s)\,\dlat s\\
&=\int_0^{s_0} \sum_{x\in(K+L+(-2,2)^n)\cap\Z^n}
\chi_{_{\d_s}}(x)\,\dlat s\\
&=\sum_{x\in(K+L+(-2,2)^n)\cap\Z^n}
\min\left\{\Gsub{n}\Bigl(\bigl(K+(-1,1)^n\bigr)\cap\bigl(x-L+(-1,1)^n\bigr)\Bigr),s_0\right\}\\
&\leq\sum_{x\in(K+L+(-2,2)^n)\cap\Z^n}
\Gsub{n}\Bigl(\bigl(K+(-1,1)^n\bigr)\cap\bigl(x-L+(-1,1)^n\bigr)\Bigr)\\
&=\sum_{x\in(K+L+(-2,2)^n)\cap\Z^n}\sum_{y\in\Z^n}\chi_{_{K+(-1,1)^n}}(y)\chi_{_{x-L+(-1,1)^n}}(y)\\
&=\sum_{x\in(K+L+(-2,2)^n)\cap\Z^n}\sum_{y\in\Z^n}\chi_{_{K+(-1,1)^n}}(y)\chi_{_{y+L+(-1,1)^n}}(x)\cr
&=\sum_{y\in\Z^n}\chi_{_{K+(-1,1)^n}}(y) \sum_{x\in(K+L+(-2,2)^n)\cap\Z^n}\chi_{_{y+L+(-1,1)^n}}(x)\cr
&=\Gsub{n}\bigl(K+(-1,1)^n\bigr)\Gsub{n}\bigl(L+(-1,1)^n\bigr).\\
\end{split}
\end{equation*}

\noindent This concludes the proof.
\end{proof}

\medskip

Next we will use a different approach to show a discrete analogue of the Rogers-Shephard inequality \eqref{e:RS_K-K_vol}. To this aim, let $K\subset\R^n$ be a non-empty convex bounded set. Then the following two relations involving the lattice point enumerator $\Gsub{n}(\cdot)$ and the volume $\vol(\cdot)$ hold:
\begin{equation}\label{e:relation_G<vol}
\Gsub{n}(K)\leq\vol\left(K+\left(-\frac{1}{2},\frac{1}{2}\right)^n\right)
\end{equation}
and
\begin{equation}\label{e:relation_vol<G}
\vol(K)\leq\Gsub{n}\left(K+\left(-\frac{1}{2},\frac{1}{2}\right)^n\right).
\end{equation}

Regarding \eqref{e:relation_G<vol} (cf.~e.g.~\cite[Equation~(3.3)]{GW}), notice that it can be easily deduced from the inclusion $K\cap\Z^n\subset K$ jointly with the fact that $\Gsub{n}(K)=\vol\bigl((K\cap\Z^n) + (-1/2,1/2)^n\bigr)$.

It is also easy to derive \eqref{e:relation_vol<G}
(although it is ``folklore'', we have not found a precise reference for it):
let $A=\bigl\{z\in\Z^n: \, \bigl(z+(-1/2,1/2)^n\bigr)\cap K\neq\emptyset\bigr\}$,
for which one clearly has
\[
A\subset\left(K+\left(-\frac{1}{2},\frac{1}{2}\right)^n\right)\cap\Z^n
\] and thus
\begin{equation}\label{e:proving_vol<G_1}
|A|\leq\Gsub{n}\left(K+\left(-\frac{1}{2},\frac{1}{2}\right)^n\right).
\end{equation}
Moreover, taking the null measure set $M:=\R^n\setminus\bigl(\Z^n+(-1/2,1/2)^n\bigr)$, we get
\[
K\subset M\cup \left[A+\left(-\frac{1}{2},\frac{1}{2}\right)^n\right]
\]
and hence
\begin{equation}\label{e:proving_vol<G_2}
\vol(K)\leq \vol\left(A+\left(-\frac{1}{2},\frac{1}{2}\right)^n\right)
=|A|.
\end{equation}
Thus, using \eqref{e:proving_vol<G_2} and \eqref{e:proving_vol<G_1}, \eqref{e:relation_vol<G} follows.

\medskip

Now, we are in a condition to show Theorems~\ref{t:RS_K-K} and \ref{t:RS_K_L}.
\begin{proof}[Proof of Theorem~\ref{t:RS_K-K}]
Using \eqref{e:relation_G<vol}
jointly with the classical Rogers-Shephard inequality \eqref{e:RS_K-K_vol}
(for which the assumption on the convex bounded set $K$ to be closed is actually not
necessary, due to the facts that the boundary of a convex set has null measure and the closure of the Minkowski
sum of bounded sets is the Minkowski sum of the closure of them)
and \eqref{e:relation_vol<G}, we get
\begin{equation*}
\begin{split}
\Gsub{n}(K-K)
&\leq \vol\left(K-K+\left(-\frac{1}{2},\frac{1}{2}\right)^n\right)\\
& = \vol\left(K+\left(-\frac{1}{4},\frac{1}{4}\right)^n-\left[K+\left(-\frac{1}{4},\frac{1}{4}\right)^n\right]\right)\\
&\leq\binom{2n}{n} \vol\left(K+\left(-\frac{1}{4},\frac{1}{4}\right)^n\right)
\leq\binom{2n}{n} \Gsub{n}\left(K+\left(-\frac{3}{4},\frac{3}{4}\right)^n\right).
\end{split}
\end{equation*}
This concludes the proof.
\end{proof}

\begin{proof}[Proof of Theorem~\ref{t:RS_K_L}]
By \eqref{e:relation_G<vol} and the fact that, for any $A,B,C\subset\R^n$,
\begin{equation}\label{e:A,B,C_inters_sum}
(A\cap B) +C\subset (A+C)\cap(B+C)
\end{equation}
(see~e.g.~\cite[Equation~(3.2)]{Sch2}), we get
\begin{equation*}
\begin{split}
\Gsub{n}(&K+L)\Gsub{n}\bigl(K\cap(-L)\bigr)\\
&\leq \vol\left(K+L+\left(-\frac{1}{2},\frac{1}{2}\right)^n\right)
\vol\left(K\cap(-L)+\left(-\frac{1}{2},\frac{1}{2}\right)^n\right)\\
&\leq c_{_{K,L}}\,\vol\bigl(K+L+(-1,1)^n\bigr)
\vol\left(\left[K+\left(-\frac{1}{2},\frac{1}{2}\right)^n\right]\cap
\left[-L+\left(-\frac{1}{2},\frac{1}{2}\right)^n\right]\right).
\end{split}
\end{equation*}
Now, applying the classical Rogers-Shephard inequality \eqref{e:RS_K_L_vol} (again, the assumption on the convex sets $K,L$ to be closed is not needed) jointly with \eqref{e:relation_vol<G}, we have
\begin{equation*}
\begin{split}
c_{_{K,L}}\,\vol&\bigl(K+L+(-1,1)^n\bigr) \vol\left(\left[K+\left(-\frac{1}{2},\frac{1}{2}\right)^n\right]
\cap\left[-L+\left(-\frac{1}{2},\frac{1}{2}\right)^n\right]\right)\\
&\leq \binom{2n}{n} c_{_{K,L}}\,\vol\left(K+\left(-\frac{1}{2},\frac{1}{2}\right)^n\right)
\vol\left(L+\left(-\frac{1}{2},\frac{1}{2}\right)^n\right)\\
&\leq \binom{2n}{n} c_{_{K,L}}\,\Gsub{n}\bigl(K+(-1,1)^n\bigr)\Gsub{n}\bigl(L+(-1,1)^n\bigr).
\end{split}
\end{equation*}
This concludes the proof.
\end{proof}

\begin{remark}\label{r:new proof Theor projsect}
We can also use this approach based on the relations between the volume and the lattice point enumerator to provide an alternative proof to
Theorem~\ref{t:RS_sect_proj}.
Indeed, from \eqref{e:relation_G<vol}, \eqref{e:A,B,C_inters_sum}, \eqref{e:RS_sect_proj_vol} (for which the assumption on the convex set $K$ to be closed is not necessary) and \eqref{e:relation_vol<G}, we obtain
\begin{equation*}
\begin{split}
\Gsub{n-k}(P_{H^\perp}& K)\Gsub{k}(K\cap H)\\
&\leq \vol_{n-k}\left(\bigl(P_{H^\perp} K\bigr)+\frac{1}{2} \Cube{H^\perp}\right)
\vol_{k}\left(\bigl(K\cap H\bigr)+\frac{1}{2} \Cube{H}\right)\\
&\leq \vol_{n-k}\left(P_{H^\perp} \left[K+\left(-\frac{1}{2},\frac{1}{2}\right)^n\right]\right)
\vol_{k}\left(\left[K+\left(-\frac{1}{2},\frac{1}{2}\right)^n\right]\cap H\right)\\
&\leq \binom{n}{k}\vol\left(K+\left(-\frac{1}{2},\frac{1}{2}\right)^n\right)
\leq \binom{n}{k}\Gsub{n}\bigl(K+(-1,1)^n\bigr).
\end{split}
\end{equation*}
We point out that, now, there is no need to assume that $0\in K$.
However, notice that such a method does not allow us to show the statement of the stronger inequality collected in Theorem \ref{t:RS_sect_proj_refined}.
\end{remark}

\smallskip

Given a convex bounded set containing the origin $K$, if we
apply \eqref{e:RS_diff_K_L_refined} to the sets $K$ and
$L':=-K+(-1,1)^n$ we have
\begin{equation}\label{e:RS_diff_K-K_refined_sums}
\sum_{i=0}^{n}\frac{n}{2n-i}\binom{n}{i}\binom{2n}{i}^{-1}
\Gsub{n}(K-K)^{i/n}
\leq\frac{\Gsub{n}\bigl(K+(-1,1)^n\bigr)\Gsub{n}\bigl(K+(-2,2)^n\bigr)}{\Gsub{n}(K)}.
\end{equation}
At this point, it is natural to compare the discrete analogues of \eqref{e:RS_K-K_vol} that involve
extensions of $K$ by Minkowski adding certain cubes, i.e., the above inequality, \eqref{e:RS_K-K} and \eqref{e:RS_K-K_2}.

First, to compare this inequality with \eqref{e:RS_K-K}, we need to relate
\begin{equation*}
\frac{\Gsub{n}\bigl(K+(-1,1)^n\bigr)\Gsub{n}\bigl(K+(-2,2)^n\bigr)}{\Gsub{n}(K)}
-\sum_{i=0}^{n-1}\frac{n}{2n-i}\binom{n}{i}\binom{2n}{i}^{-1}
\Gsub{n}(K-K)^{i/n}
\end{equation*}
and
\begin{equation*}
\Gsub{n}\left(K+\left(-\frac{3}{4},\frac{3}{4}\right)^n\right).
\end{equation*}

Although, unfortunately, we do not have a full answer to this question, next we show that in dimension
$n=2$ the latter expression provides us with a smaller
upper bound for
\[
\Gsub{n}(K-K)\binom{2n}{n}^{-1}
\]
and hence, in the plane, \eqref{e:RS_K-K} is tighter than \eqref{e:RS_diff_K-K_refined_sums}.
This is an immediate consequence of the following result.

\begin{proposition}
Let $K\in\K^2$ be a planar convex body containing the origin. Then
\begin{equation}\label{e:comparing_disc_RS}
\Gsub{2}\bigl(K+(-1,1)^2\bigr)
<\Gsub{2}\bigl(K+(-2,2)^2\bigr)
-\sum_{i=0}^{1}\frac{2}{4-i}\binom{2}{i}\binom{4}{i}^{-1}
\Gsub{2}(K-K)^{i/2}.
\end{equation}
\end{proposition}

\begin{proof}
Let $H_i=\{x\in\R^2:\, \esc{x, e_i}=0\}$, $i=1,2$,
and set
\[
m:=\max_{i=1,2}\Gsub{1}(P_{H_i}K),
\]
for which we will assume without loss of generality that $m=\Gsub{1}(P_{H_1}K)$.
Then, $K$ is contained in a rectangle $[-a_1,b_1]\times[-a_2,b_2]$, with $a_i,b_i\geq0$ and $\Gsub{1}\bigl([-a_i,b_i]\bigr)\leq m$, $i=1,2$.

\smallskip

So, we clearly have that $\Gsub{2}(K)\leq m^2$ and $\Gsub{2}(K-K)\leq(2m+1)^2$. Moreover,
since $K+(-1,1)^2$ is open and thus, for any $x\in (P_{H_1}K)\cap\Z^2$,
\[
\bigl(K+(-2,2)^2\bigr)\cap \bigl(x+\lin\{e_1\}\bigr)
\]
contains at least two more integer points than
\[
\bigl(K+(-1,1)^2\bigr)\cap \bigl(x+\lin\{e_1\}\bigr),
\]
we get
\[
\Gsub{2}\bigl(K+(-2,2)^2\bigr)\geq \Gsub{2}\bigl(K+(-1,1)^2\bigr) + 2m.
\]
Altogether, since $m\geq1$ (because $0\in K$), we have
\[
\frac{1}{2}+\frac{1}{3}\Gsub{2}(K-K)^{1/2}\leq\frac{1}{2}+\frac{1}{3}(2m+1)<2m
\leq \Gsub{2}\bigl(K+(-2,2)^2\bigr)-\Gsub{2}\bigl(K+(-1,1)^2\bigr),
\]
which is equivalent to \eqref{e:comparing_disc_RS}.
\end{proof}
Next we relate \eqref{e:RS_K-K} and \eqref{e:RS_K-K_2}.
\begin{remark}
Inequalities \eqref{e:RS_K-K} and \eqref{e:RS_K-K_2} are not comparable. Indeed, taking $K=[-r,r]^n$ with $r>0$, one has
\[
c_{_{K,-K}}=\left(\frac{4r+1}{4r+2}\right)^n<1.
\]
So, on the one hand, if $r\in\N$ we get
\[\Gsub{n}(K)=\Gsub{n}\left(K+\left(-\frac{3}{4},\frac{3}{4}\right)^n\right)=\Gsub{n}\bigl(K+(-1,1)^n\bigr)\]
and thus
\[
c_{_{K,-K}}\frac{\Gsub{n}\bigl(K+(-1,1)^n\bigr)^2}{\Gsub{n}(K)}
<\Gsub{n}\left(K+\left(-\frac{3}{4},\frac{3}{4}\right)^n\right).
\]
On the other hand, if $r\notin\N$ then
\[
c_{_{K,-K}}\frac{\Gsub{n}\bigl(K+(-1,1)^n\bigr)}{\Gsub{n}(K)}
=\frac{(4r+1)^n\bigl(2\floor{r}+3\bigr)^n}{(4r+2)^n\bigl(2\floor{r}+1\bigr)^n}
=\left(\frac{1+\frac{2}{2\floor{r}+1}}{1+\frac{1}{4r+1}}\right)^n>1,
\]
where $\floor{r}$ denotes the floor function of $r$ (i.e., the greatest
integer less than or equal to $r$), and hence
\[
c_{_{K,-K}}\frac{\Gsub{n}\bigl(K+(-1,1)^n\bigr)^2}{\Gsub{n}(K)}
>\Gsub{n}\left(K+\left(-\frac{3}{4},\frac{3}{4}\right)^n\right).
\]
\end{remark}

\subsection{From the discrete analogues to the continuous versions}

It is intuitive that one can approximate the volume of a convex body by successively shrinking the lattice.
This can be easily seen by means of the fact that the volume and the lattice point enumerator are equivalent when
the convex body is ``large enough''.
More precisely, given a convex body $K\in\K^n$ with dimension $\dim K=n$, we have
\begin{equation}\label{e:limitrelation_Gn_vol}
\lim_{r\to\infty}\frac{\Gsub{n}(rK)}{r^n}=\vol(K)
\end{equation}
(see~e.g.~\cite[Lemma~3.22]{TV}).
Moreover, it is easy to check that
\begin{equation}\label{e:limitrelation_Gn_vol_2}
\lim_{r\to\infty}\frac{\Gsub{n}(rK+M)}{r^n}=\vol(K)
\end{equation}
for any bounded convex set $M$ containing the origin. Indeed, given $\varepsilon>0$
it follows that, for any $r>0$ large enough, $M\subset (r\varepsilon K)+z_r$ for some $z_r\in\Z^n$, and
thus
\begin{equation*}
\begin{split}
\vol(K)=\lim_{r\to\infty}\frac{\Gsub{n}(rK)}{r^n}
&\leq\liminf_{r\to\infty}\frac{\Gsub{n}(rK+M)}{r^n}
\leq\limsup_{r\to\infty}\frac{\Gsub{n}(rK+M)}{r^n}\\
&\leq \lim_{r\to\infty}\frac{\Gsub{n}\bigl(r(K+\varepsilon K)\bigr)}{r^n}
= \vol(K+\varepsilon K)=(1+\varepsilon)^n\vol(K).
\end{split}
\end{equation*}
Since $\varepsilon>0$ was arbitrary, \eqref{e:limitrelation_Gn_vol_2} holds.

\smallskip

We then conclude this section by proving that the discrete versions of the projection-section and the Rogers-Shephard inequalities we have previously shown imply their corresponding continuous analogues, by exploiting the above relations between the lattice point enumerator and the volume.
To this aim, regarding the discrete projection-section type inequalities, we will show that \eqref{e:RS_sect_proj}
already implies \eqref{e:RS_sect_proj_vol} (and hence, the same is obtained from the stronger inequality \eqref{e:RS_sect_proj_refined}). In the same way, we will prove that \eqref{e:RS_K_L_weak} is enough to derive \eqref{e:RS_K_L_vol} (and thus, the same happens for the more powerful inequalities \eqref{e:RS_K_L} and \eqref{e:RS_diff_K_L_refined}). Moreover, in particular, \eqref{e:RS_K-K_weak} implies \eqref{e:RS_K-K_vol} (and so, the same is true for the stronger versions \eqref{e:RS_K-K}, \eqref{e:RS_K-K_2} and \eqref{e:RS_diff_K-K_refined}).
\begin{theorem}\label{t:RS_disc_imply_cont}
Let $K,L\in\K^n$ be convex bodies containing the origin with $\dim K=\dim L=n$. Then
\begin{enumerate}
\item The discrete inequality \eqref{e:RS_sect_proj} for the lattice point enumerator implies the classical projection-section inequality \eqref{e:RS_sect_proj_vol} for the volume.
\smallskip
\item The discrete inequality \eqref{e:RS_K_L_weak} for the lattice point enumerator
implies the classical Rogers-Shephard inequality \eqref{e:RS_K_L_vol} for the volume.
\end{enumerate}
\end{theorem}

\begin{proof}
Applying
\eqref{e:RS_sect_proj} with $rK$ (for $r>0$), taking limits as $r\to\infty$
and using \eqref{e:limitrelation_Gn_vol} and \eqref{e:limitrelation_Gn_vol_2}, we get
\[
\begin{split}
\vol_{n-k}(P_{H^\perp} K)\vol_k(K\cap H)
&=\lim_{r\to\infty}
\frac{\Gsub{n-k}(rP_{H^\perp} K)}{r^{n-k}}\cdot\frac{\Gsub{k}\bigl(r(K\cap H)\bigr)}{r^k}\\
&=\lim_{r\to\infty}\frac{\Gsub{n-k}\bigl(P_{H^\perp} (rK)\bigr)\Gsub{k}\bigl((rK)\cap H\bigr)}{r^n}\\
&\leq\lim_{r\to\infty} \binom{n}{k}\frac{\Gsub{n}\bigl(rK+(-1,1)^n\bigr)}{r^n}
=\binom{n}{k}\vol(K).
\end{split}
\]
Analogously, but now applying \eqref{e:RS_K_L_weak} with $rK$ and $rL$ (for $r>0$), we obtain
\[
\begin{split}
\vol(K+L)\vol\bigl(K\cap(-L)\bigr)
&=\lim_{r\to\infty}
\frac{\Gsub{n}\bigl(r(K+L)\bigr)\Gsub{n}\Bigl(r\bigl(K\cap(-L)\bigr)\Bigr)}{r^{2n}}\\
&=\lim_{r\to\infty}\frac{\Gsub{n}(rK+rL)\Gsub{n}\bigl((rK)\cap(-rL)\bigr)}{r^{2n}}\\
&\leq \lim_{r\to\infty}\binom{2n}{n}\frac{\Gsub{n}\bigl(rK+(-1,1)^n\bigr)
\Gsub{n}\bigl(rL+(-2,2)^n\bigr)}{r^{2n}}\\
&=\binom{2n}{n}\vol(K)\vol(L).
\end{split}
\]
This concludes the proof.
\end{proof}

\begin{remark}\label{r:sharp}
Since \eqref{e:RS_sect_proj_vol} and \eqref{e:RS_K_L_vol} are sharp, from the proof above, we have that their discrete analogues \eqref{e:RS_sect_proj} and \eqref{e:RS_K_L_weak} (and hence their corresponding stronger related versions) are asymptotically sharp.
\end{remark}

\section{A discrete analogue of Berwald's inequality}\label{s:Berwald}

Let $K\subset\R^n$ be a convex bounded set containing the origin, let $f:K\longrightarrow\R_{\geq0}$ be a non-negative function and set $p>0$. We will write $\mu$ to denote the \emph{counting measure} on $\Z^n$, considered as a measure on $\R^n$, namely, the measure on $\R^n$ given by $\mu(M)=\Gsub{n}(M)$ for any $M\subset\R^n$. First we observe that we have
\begin{equation}\label{e: sum f(x)_as integral}
\sum_{x\in K\cap\Z^n} f(x)^p=\int_0^\infty pt^{p-1}\Gsub{n}\Bigl(\bigl\{x\in K: \, f(x)>t\bigr\}\Bigr)\,\dlat t.
\end{equation}
Indeed, by Fubini's theorem, we obtain
\begin{equation*}
\begin{split}
\sum_{x\in K\cap\Z^n} f(x)^p
=\int_{\R^n} f(x)^p \chi_{_K}(x)\,\dlat \mu(x)
&=\int_{\R^n} \left(\int_0^{f(x)} pt^{p-1}\,\dlat t\right) \chi_{_K}(x)\,\dlat \mu(x)\\
&=\int_0^\infty \int_{\R^n} pt^{p-1} \chi_{_K}(x) \chi_{_{(0,f(x))}}(t) \,\dlat \mu(x) \,\dlat t\\
&=\int_0^\infty  pt^{p-1} \int_{\R^n} \chi_{_{\{x\in K: \, f(x)>t\}}}(x) \,\dlat \mu(x) \,\dlat t\\
&=\int_0^\infty pt^{p-1}\Gsub{n}\Bigl(\bigl\{x\in K: \, f(x)>t\bigr\}\Bigr)\,\dlat t,
\end{split}
\end{equation*}
which shows \eqref{e: sum f(x)_as integral}.

\smallskip

To prove Theorem~\ref{t:Berwald}, we need the following auxiliary results.

\begin{lemma}\label{l:sum_h_m geq m}
Let $K\subset\R^n$ be a convex bounded set containing the origin and let $m>0$.
Let $h_m:K+(-1,1)^n\longrightarrow\R_{\geq0}$
be the concave function whose hypograph is the closure of
$\conv\bigl((K\times\{0\})\cup(0_{n},m)\bigr) +\bigl((-1,1)^n\times\{0\}\bigr)$.
Then, for every $p>0$,
\begin{equation}\label{e:sum_h_m geq m}
\left(\frac{\binom{n+p}{n}}{\Gsub{n}(K)}\sum_{x\in(K+(-1,1)^n)\cap\Z^n} h_m(x)^p\right)^{1/p} \geq m.
\end{equation}
\end{lemma}
\begin{proof}
Observe that, for any $0\leq t< m$,
\[
\begin{split}
\Bigl(\conv\bigl((K\times\{0\})\cup(0_{n},m)\bigr) + \bigl((-1,1)^n \times \{0\}\bigr)\Bigr)
\cap\bigl(\R^n\times\{t\}\bigr)\\
=\left(\left(1-\frac{t}{m}\right)K+(-1,1)^n\right)\times\{t\}
\end{split}
\]
and thus
\[
\bigl\{x\in K+(-1,1)^n: \, h_m(x)> t\bigr\}=\left(1-\frac{t}{m}\right)K+(-1,1)^n.
\]
Then, using \eqref{e: BM_lattice_point_no_G(K)G(L)>0} we get
\begin{equation*}
\begin{split}
\int_0^\infty & p t^{p-1}\Gsub{n}\Bigl(\bigl\{x\in K+(-1,1)^n: \, h_m(x)> t\bigr\}\Bigr)\,\dlat t\\
&= \int_0^m p t^{p-1}\Gsub{n}\left(\left(1-\frac{t}{m}\right)K+(-1,1)^n\right)\,\dlat t\\
&\geq \int_0^m p t^{p-1}\left(1-\frac{t}{m}\right)^n\Gsub{n}(K)\,\dlat t
= p \,m^p\, \Gsub{n}(K)\int_0^1 s^{p-1}(1-s)^n\,\dlat s\\
&= p \,m^p\, \Gsub{n}(K) \frac{\Gamma(p)\Gamma(n+1)}{\Gamma(n+p+1)}= m^p \, \Gsub{n}(K)\binom{n+p}{n}^{-1}.
\end{split}
\end{equation*}
This, together with \eqref{e: sum f(x)_as integral} applied to the function $h_m$, yields
\[
\sum_{x\in(K+(-1,1)^n)\cap\Z^n} h_m(x)^p \geq m^p \, \Gsub{n}(K)\binom{n+p}{n}^{-1},
\]
which shows \eqref{e:sum_h_m geq m}.
\end{proof}

\smallskip

Now, given a concave function $f:K\longrightarrow\R_{\geq0}$ defined on a convex bounded set $K\subset\R^n$, we will relate the number of integer points of the superlevel sets of both the function $f$ and its extension $f^{\symbol}$ (whose hypograph is the closure of the Minkowski addition of the hypograph of $f$ and $(-1,1)^n\times\{0\}$) in terms of a suitable $(1/n)$-concave function (on its support).

\begin{lemma}\label{l:ineqs_g(t)_t_0}
Let $K\subset\R^n$ be a convex bounded set containing the origin
and let $f:K\longrightarrow\R_{\geq0}$ be a concave function with $f(0)=|f|_\infty>0$.
For any $p>0$, let
\begin{equation*}
m=\left(\frac{\binom{n+p}{n}}{\Gsub{n}(K)}\sum_{x\in(K+(-1,1)^n)\cap\Z^n}
\bigl(f^{\symbol}\bigr)^{p}(x)\right)^{1/p}
\end{equation*}
and let $g:\R_{\geq0}\longrightarrow\R_{\geq0}$ be the function given by
\[
g(t)=\left\{
\begin{array}{ll}
\left(1-\frac{t}{m}\right)^n \Gsub{n}(K) & \text{ if } t\leq m,\\[2mm]
0 & \text{ otherwise}.
\end{array}
\right.
\]
Then, there exists $t_0\in\R_{\geq0}$ such that
\begin{equation}\label{e:ineq_g(t)_1}
\Gsub{n}\Bigl(\bigl\{x\in K+(-1,1)^n: \, f^{\symbol}(x)> t \bigr\}\Bigr) > g(t)
\end{equation}
for all $0\leq t<t_0$ and
\begin{equation}\label{e:ineq_g(t)_2}
g(t)\geq \Gsub{n}\Bigl(\bigl\{x\in K: \, f(x) > t\bigr\}\Bigr)
\end{equation}
for all $t\geq t_0$.
\end{lemma}

\begin{proof}
First we will show that $m\geq |f|_{\infty}$. To this aim, assume by contradiction that $m<|f|_{\infty}$ and
set $h_m:K+(-1,1)^n\longrightarrow\R_{\geq0}$
the concave function whose hypograph is the closure of
\[\conv\bigl((K\times\{0\})\cup(0_{n},m)\bigr) +\bigl((-1,1)^n\times\{0\}\bigr).\]
Then, by the concavity of $f$,
\[
\hyp(f)\supset\conv\bigl((K\times\{0\})\cup(0_{n},m)\bigr)
\]
and so
$\hyp\bigl(f^{\symbol}\bigr)\supset\hyp(h_m)$. This also implies that
\begin{equation*}
\bigl\{x\in K+(-1,1)^n: \, f^{\symbol}(x) > t\bigr\}
\supset\bigl\{x\in K+(-1,1)^n: \, h_m(x) > t\bigr\}
\end{equation*}
for all $0\leq t<|f|_\infty$.
Therefore, assuming that $m<|f|_{\infty}$, the latter inclusion jointly with
\eqref{e: sum f(x)_as integral} applied to the functions
$f^{\symbol}$ and $h_m$, Lemma \ref{l:sum_h_m geq m} and the fact that
\[
\Gsub{n}\Bigl(\bigl\{x\in K+(-1,1)^n: \, \, f^{\symbol}(x) > t\bigr\}\Bigr)\geq 1
\]
for all $0\leq t<|f|_\infty$ (since $f(0)=|f|_\infty$), imply that

\begin{equation*}
\begin{split}
m &= \left(\frac{\binom{n+p}{n}}{\Gsub{n}(K)}
\int_0^{|f|_\infty} p t^{p-1}\Gsub{n}\Bigl(\bigl\{x\in K+(-1,1)^n: \, f^{\symbol}(x) > t\bigr\}\Bigr)
\, \dlat t\right)^{1/p}\\
&>\left(\frac{\binom{n+p}{n}}{\Gsub{n}(K)}
\int_0^{m} p t^{p-1}
\Gsub{n}\Bigl(\bigl\{x\in K+(-1,1)^n: \, f^{\symbol}(x) > t\bigr\}\Bigr)\,\dlat t\right)^{1/p}\\
&\geq \left(\frac{\binom{n+p}{n}}{\Gsub{n}(K)}
\int_0^{m} p t^{p-1}
\Gsub{n}\Bigl(\bigl\{x\in K+(-1,1)^n: \, h_{m}(x) > t\bigr\}\Bigr)\,\dlat t\right)^{1/p}\\
&= \left(\frac{\binom{n+p}{n}}{\Gsub{n}(K)}\sum_{x\in(K+(-1,1)^n)\cap\Z^n} h_m(x)^p\right)^{1/p}
\geq m,
\end{split}
\end{equation*}
a contradiction.

Now, since $f^{\symbol} \leq |f|_\infty$, we trivially have
\[0=\Gsub{n}\Bigl(\bigl\{x\in K+(-1,1)^n: \, f^{\symbol}(x) > |f|_\infty\bigr\}\Bigr)
\leq g\bigl(|f|_\infty\bigr)
\]
and thus we may consider
\[
t_0:=\inf\left\{t>0: \, \Gsub{n}\Bigl(\bigl\{x\in K+(-1,1)^n: \, f^{\symbol}(x) > t\bigr\}\Bigr)
\leq g(t)\right\}<\infty.
\]
Then, on the one hand,
we obtain by the definition of $t_0$ that \eqref{e:ineq_g(t)_1} holds for all $0\leq t<t_0\leq|f|_\infty$.
Moreover, since
\[
\bigl\{x\in K+(-1,1)^n: \, f^{\symbol}(x) > t \bigr\}=
\bigl\{x\in K: \, f(x) > t \bigr\}+(-1,1)^n
\]
(which is an open convex bounded set) for all $0\leq t<|f|_\infty$,
we have that the function
$t\mapsto \Gsub{n}\Bigl(\bigl\{x\in K+(-1,1)^n: \, f^{\symbol}(x) > t\bigr\}\Bigr)$
is continuous from the right on $\R_{\geq0}$. Therefore, we obtain that
\[
\Gsub{n}\Bigl(\bigl\{x\in K+(-1,1)^n: \, f^{\symbol}(x) > t_0 \bigr\}\Bigr)
\leq g(t_0).
\]
On the other hand, given $t\in\bigl[t_0,|f|_\infty\bigr]$
and taking $\lambda\in(0,1]$ such that $t_0=\lambda t$, from \eqref{e: BM_lattice_point_no_G(K)G(L)>0} we obtain
\begin{equation}\label{e:proving_lemma_t0_BM}
\begin{split}
\Gsub{n}\Bigl(\bigl\{x&\in K+(-1,1)^n: \, f^{\symbol}(x) > t_0\bigr\}\Bigr)^{1/n}\\
&= \Gsub{n}\Bigl(\bigl\{x\in K: \, f(x) > t_0 \bigr\}+(-1,1)^n\Bigr)^{1/n}\\
&\geq \Gsub{n}\Bigl( (1-\lambda)\bigl\{x\in K: \, f(x) \geq 0\bigr\}+
\lambda\bigl\{x\in K: \, f(x) > t\bigr\} + (-1,1)^n \Bigr)^{1/n}\\
&\geq (1-\lambda)\Gsub{n}(K)^{1/n}+ \lambda \Gsub{n}\Bigl(\bigl\{ x\in K: \, f(x) > t \bigr\}\Bigr)^{1/n}
\end{split}
\end{equation}
and also (taking into account that
$t_0\leq|f|_\infty\leq m$)
\begin{equation}\label{e:proving_lemma_t0_g(t0)}
g(t_0)^{1/n}=
\left(1-\frac{t_0}{m}\right) \Gsub{n}(K)^{1/n}
=(1-\lambda) \Gsub{n}(K)^{1/n} + \lambda\left(1-\frac{t}{m}\right) \Gsub{n}(K)^{1/n}.
\end{equation}
Thus, using \eqref{e:proving_lemma_t0_BM} and \eqref{e:proving_lemma_t0_g(t0)},
we get that \eqref{e:ineq_g(t)_2} holds for all $t_0\leq t\leq|f|_\infty$.
This concludes the proof, since \eqref{e:ineq_g(t)_2} is further trivially true for any
$t\in\bigl[|f|_\infty,\infty\bigr)$.
\end{proof}

\begin{proof}[Proof of Theorem~\ref{t:Berwald}]
We may assume, without loss of generality, that $|f|_\infty>0$, and let $m$ and $g$ be defined as in Lemma \ref{l:ineqs_g(t)_t_0}. Observe also that, for any $r>0$,
\begin{equation}\label{e:m_integral}
\left(\frac{\binom{n+r}{n}}{\Gsub{n}(K)}\int_0^{m} r t^{r-1}\left(1-\frac{t}{m}\right)^n
\Gsub{n}(K)\,\dlat t\right)^{1/r}= m.
\end{equation}
From \eqref{e: sum f(x)_as integral}
applied to $f^{\symbol}$ jointly with the definition of $g$ and $m$,
the latter implies, in particular, that
\begin{equation}\label{e:int_0^infty t^{p-1} g(t)}
\int_{0}^{\infty}t^{p-1}\Gsub{n}\Bigl(\bigl\{x\in K+(-1,1)^n\,:\,f^{\symbol}(x)> t\bigr\}\Bigr)\,\dlat t = \int_{0}^{\infty}t^{p-1} g(t)\,\dlat t.
\end{equation}

Hence, with $t_0$ as provided by Lemma \ref{l:ineqs_g(t)_t_0} we obtain,
from \eqref{e:ineq_g(t)_1} and \eqref{e:ineq_g(t)_2}, that
\begin{equation*}
\begin{split}
\int_0^{t_0}& t^{q-1}\Bigl[\Gsub{n}\Bigl(\bigl\{x\in K+(-1,1)^n: \, f^{\symbol}(x) > t \bigr\}\Bigr)-g(t)\Bigr]\,\dlat t\cr
&\qquad - \int_{t_0}^{\infty}t^{q-1}\left[g(t)-\Gsub{n}\Bigl(\bigl\{x\in K: \, f(x) > t \bigr\}\Bigr)\right]\,\dlat t\\
&= \int_0^{t_0}t^{p-1}t^{q-p}
\left[\Gsub{n}\Bigl(\bigl\{x\in K+(-1,1)^n: \, f^{\symbol}(x) > t \bigr\}\Bigr)-g(t)\right]\,\dlat t\\
&\qquad -\int_{t_0}^{\infty}t^{p-1}t^{q-p}
\left[g(t)-\Gsub{n}\Bigl(\bigl\{x\in K: \, f(x) > t\bigr\}\Bigr)\right]\,\dlat t\\
&\leq t_0^{q-p}\int_0^{t_0}t^{p-1}
\left[\Gsub{n}\Bigl(\bigl\{x\in K+(-1,1)^n: \, f^{\symbol}(x) > t\bigr\}\Bigr)-g(t)\right]\,\dlat t\\
&\qquad - t_0^{q-p}\int_{t_0}^{\infty}t^{p-1}
\left[g(t)-\Gsub{n}\Bigl(\bigl\{x\in K: \, f(x) > t\bigr\}\Bigr)\right]\,\dlat t.
\end{split}
\end{equation*}
Moreover, we have
\begin{equation*}
\begin{split}
t_0^{q-p}\int_0^{t_0}& t^{p-1}\Bigl[\Gsub{n}\Bigl(\bigl\{x\in K+(-1,1)^n: \, f^{\symbol}(x)> t\bigr\}\Bigr)-g(t)\Bigr]\,\dlat t\\
&\qquad + t_0^{q-p}\int_{t_0}^{\infty}t^{p-1}
\Bigl[\Gsub{n}\Bigl(\bigl\{x\in K: \, f(x)> t \bigr\}\Bigr)-g(t)\Bigr]\,\dlat t\\
&\leq t_0^{q-p}\int_{0}^{\infty}t^{p-1}\Bigl[\Gsub{n}\Bigl(\bigl\{x\in K+(-1,1)^n\,:\,f^{\symbol}(x)> t\bigr\}\Bigr)-g(t)\Bigr]\,\dlat t = 0,
\end{split}
\end{equation*}
where the latter equality follows from \eqref{e:int_0^infty t^{p-1} g(t)}.

Altogether, we have shown that
\begin{equation*}
\begin{split}
 \int_0^{t_0}& t^{q-1}
\Gsub{n}\Bigl(\bigl\{x\in K+(-1,1)^n: \, f^{\symbol}(x) > t\bigr\}\Bigr)\,\dlat t\\
&\qquad + \int_{t_0}^\infty t^{q-1}\Gsub{n}\Bigl(\bigl\{x\in K: \, f(x) > t\bigr\}\Bigr)\,\dlat t
\leq \int_0^\infty t^{q-1}g(t)\,\dlat t
\end{split}
\end{equation*}
and hence
\[
\int_{0}^\infty  t^{q-1}\Gsub{n}\Bigl(\bigl\{x\in K: \, f(x)> t\bigr\}\Bigr)\,\dlat t
\leq\int_0^\infty  t^{q-1}g(t)\,\dlat t.
\]
Consequently, from \eqref{e:m_integral} for $r=q$, we have
\begin{equation*}
\begin{split}
& \left(\frac{\binom{n+q}{n}}{\Gsub{n}(K)}\int_0^\infty q t^{q-1}
\Gsub{n}\Bigl(\bigl\{x\in K: \, f(x)> t \bigr\}\Bigr)\,\dlat t\right)^{1/q}\\
& \qquad \leq \left(\frac{\binom{n+q}{n}}{\Gsub{n}(K)}\int_0^{m} q t^{q-1}\left(1-\frac{t}{m}\right)^n
\Gsub{n}(K)\,\dlat t\right)^{1/q}= m
\end{split}
\end{equation*}
and thus, from \eqref{e: sum f(x)_as integral}
applied to $f$ and $q$,
\[
\left(\frac{\binom{n+q}{n}}{\Gsub{n}(K)}\sum_{x\in K\cap\Z^n}f^{q}(x)\right)^{1/q}
\leq\left(\frac{\binom{n+p}{n}}{\Gsub{n}(K)}\sum_{x\in(K+(-1,1)^n)\cap\Z^n}\bigl(f^{\symbol}\bigr)^{p}(x)\right)^{1/p}.
\]
This concludes the proof.
\end{proof}

\smallskip

As briefly pointed out within the introduction, the continuous version of Berwald's inequality (Theorem \ref{t:Berwald_classical}) allows us to derive the Rogers-Shephard inequalities \eqref{e:RS_sect_proj_vol} and \eqref{e:RS_K-K_vol}. To show this, first notice that
Stirling's formula for the gamma function yields the asymptotic formula
\[\lim_{x\to\infty}\frac{\Gamma(x)}{\sqrt{\frac{2\pi}{x}}\left(\frac{x}{e}\right)^x}=1,\]
which implies, in particular, that
$\binom{n+q}{n}^{1/q}\to 1$ as $q\to\infty$. Moreover,
given a convex body $K\in\K^n$ with $\dim K=n$
and a concave function $f:K\longrightarrow\R_{\geq0}$,
it is well-known that
\[\lim_{q\to\infty}\left(\int_{K}f^{q}(x)\,\dlat x\right)^{1/q}=|f|_\infty\] (here we notice that,
since $f$ is concave, $|f|_{\infty}$ agrees with $\esssup_{x\in K}f(x)$). Thus, applying Theorem \ref{t:Berwald_classical}
with $p=k$ (and $n'=n-k$) and the concave function (cf. \eqref{e:BM}) $f:P_{H^\perp}K\longrightarrow\R_{\geq0}$ given by \[f(x)=\vol_k\bigl(K\cap(x+H\bigr)\bigr)^{1/k}\]
for $H\in\L^n_k$, one gets Theorem \ref{t:RS_sect_proj_vol}
by taking limit as $q\to\infty$.
Indeed, one has
\begin{equation*}
\begin{split}
\vol_k(K\cap H)^{1/k}&\leq|f|_\infty=
\lim_{q\to\infty}\left(\frac{\binom{n-k+q}{n-k}}{\vol_{n-k}(P_{H^\perp}K)}
\int_{P_{H^\perp}K}f^{q}(x)\,\dlat x\right)^{1/q}\\
&\leq\left(\frac{\binom{n}{k}}{\vol_{n-k}(P_{H^\perp}K)}\int_{P_{H^\perp}K}f^{k}(x)\,\dlat x\right)^{1/k}\\
&=\left(\frac{\binom{n}{k}}{\vol_{n-k}(P_{H^\perp}K)}\vol(K)\right)^{1/k},
\end{split}
\end{equation*}
where the last equality follows from Fubini's theorem.

Analogously, from Theorem \ref{t:Berwald_classical} for $p=n$ and the concave function (cf. \eqref{e:BM}) $f:K-K\longrightarrow\R_{\geq0}$ given by \[f(x)=\vol\bigl(K\cap(x+K)\bigr)^{1/n},\]
for which one has
\begin{equation*}
\begin{split}
\int_{K-K}f^{n}(x)\,\dlat x&=\int_{\R^n}\int_{\R^n}\chi_{_K}(y)\chi_{_{x+K}}(y)\,\dlat y\,\dlat x
=\int_{\R^n}\int_{\R^n}\chi_{_K}(y)\chi_{_{y-K}}(x)\,\dlat x\,\dlat y\\
&=\vol(K)^2,
\end{split}
\end{equation*}
one gets
\begin{equation*}
\begin{split}
\vol(K)^{1/n}&=|f|_\infty=
\lim_{q\to\infty}\left(\frac{\binom{n+q}{n}}{\vol(K-K)}
\int_{K-K}f^{q}(x)\,\dlat x\right)^{1/q}\\
&\leq\left(\frac{\binom{2n}{n}}{\vol(K-K)}\int_{K-K}f^{n}(x)\,\dlat x\right)^{1/n}=
\left(\frac{\binom{2n}{n}}{\vol(K-K)}\vol(K)^2\right)^{1/n},
\end{split}
\end{equation*}
and so Theorem \ref{t:RS_vol} follows.

Arguing in a similar way in the discrete setting, but now applying Theorem \ref{t:Berwald} (for the above-mentioned functions and values of $p$, and letting $q\to\infty$), we get the following results:

\begin{corollary}
Let $k\in\{1,\dots,n-1\}$ and $H\in\L^n_{k}$.
Let $K\subset\R^n$ be a convex bounded set containing the origin. Then
\[
\Gsub{n-k}(P_{H^\perp}K)\vol_k(K\cap H)
\leq \binom{n}{k}\sum_{x\in (P_{H^\perp}K+\Cube{H^\perp})\cap\Z^n}
\,\sup_{z\in \Cube{H^\perp}}\vol_k\Bigl(K\cap\bigl((x+z)+H\bigr)\Bigr).
\]
\end{corollary}


\begin{corollary}
Let $K\subset\R^n$ be a convex bounded set containing the origin. Then
\[
\Gsub{n}(K-K)\vol(K)
\leq \binom{2n}{n}\sum_{x\in (K-K+(-1,1)^n)\cap\Z^n}
\,\sup_{z\in(-1,1)^n}\vol\Bigl(K\cap\bigl((x+z)+K\bigr)\Bigr).
\]
\end{corollary}
We point out that the more general result, involving convex bounded sets $K,L\subset\R^n$ containing the origin such that $\max_{x\in K+L}\vol\bigl(K\cap(x-L)\bigr)=\vol\bigl(K\cap(-L)\bigr)$ may be also derived, obtaining
in this way that
\[
\Gsub{n}(K+L)\vol\bigl(K\cap(-L)\bigr)
\leq \binom{2n}{n}\sum_{x\in (K+L+(-1,1)^n)\cap\Z^n}
\,\sup_{z\in(-1,1)^n}\vol\Bigl(K\cap\bigl((x+z)-L\bigr)\Bigr).
\]
We also observe that one cannot immediately derive, in principle, other discrete versions of the Rogers-Shephard inequalities \eqref{e:RS_sect_proj_vol} and \eqref{e:RS_K-K_vol} from Theorem \ref{t:Berwald}, despite counting with the discrete analogue \eqref{e: BM_lattice_point_no_G(K)G(L)>0} of the classical Brunn-Minkowski inequality, because of the lack of concavity of the functional $\Gsub{n}(\cdot)^{1/n}$. This is the reason for which Theorem \ref{t:Berwald} yields the above discrete counterparts of \eqref{e:RS_sect_proj_vol} and \eqref{e:RS_K-K_vol}, where the volume arises jointly with the lattice point enumerator. Some engaging examples of discrete analogues of classical inequalities where these two functionals appear together can be found in \cite{AHZ}.

\smallskip

We conclude the paper by showing that the discrete version of Berwald's inequality we have previously proven implies its continuous analogue:
\begin{theorem}\label{t:Berw_disc_imply_cont}
Let $K\in\K^n$ and let $f:K\longrightarrow\R_{\geq0}$ be a concave function. Then
the discrete inequality \eqref{e:Berwald} implies the classical Berwald inequality \eqref{e:Berwald_classical}.
\end{theorem}

Before proving this result we observe the following.
Given $K\in\K^n$ and a concave function $f:K\longrightarrow\R_{\geq0}$,
we have
\begin{equation}\label{e:f_Riem_sums}
\lim_{r\to\infty}\left[\frac{1}{r^n}\sum_{x\in(rK)\cap\Z^n}f\left(\frac{x}{r}\right)\right]
=\lim_{r\to\infty}\left[\frac{1}{r^n}\sum_{y\in K\cap((1/r)\Z^n)}f(y)\right]=\int_{K}f(x)\,\dlat x,
\end{equation}
since $f$ is Riemann integrable (because it is concave on the convex set $K$, whose boundary has
null measure).

Moreover,
we may assume without loss of generality that $f$ is upper
semicontinuous. Indeed, otherwise we would work with its upper closure,
which is determined via the closure of the superlevel sets of $f$ (see
\cite[page~14 and Theorem~1.6]{RoWe}) and thus has the same integral on
$\R^n$ because of Fubini's theorem together with the facts that all the
superlevel sets of $f$ are convex (since $f$ is concave) and the
boundary of a convex set has null (Lebesgue) measure.

Notice then that, for any decreasing sequence $(r_k)_{k\in\N}$ with $r_k\to 0$ as $k\to\infty$, we have
\begin{equation}\label{e:inter_suplevelsets+cubes}
\bigcap_{k=1}^{\infty}\Bigl(\bigl\{x\in K: \, f(x)\geq t\bigr\} + r_k(-1,1)^n\Bigr)
=\bigl\{x\in K: \, f(x)\geq t\bigr\}
\end{equation}
due to the fact that $\{x\in K: \, f(x)\geq t\bigr\}$ is closed for all $t\geq0$.

\begin{proof}[Proof of Theorem \ref{t:Berw_disc_imply_cont}]
On the one hand, from \eqref{e:Berwald} applied to the function $h:rK\longrightarrow\R_{\geq0}$ given by $h(x):=f(x/r)$, we get
\begin{equation*}\label{proving_Berw_disc_to_cont_1}
\left(\frac{\binom{n+q}{n}}{\Gsub{n}(rK)}\sum_{y\in(rK)\cap\Z^n}h^{q}(y)\right)^{1/q}
\leq\left(\frac{\binom{n+p}{n}}{\Gsub{n}(rK)}\sum_{y\in(rK+(-1,1)^n)\cap\Z^n}
\bigl(h^{\symbol}\bigr)^{p}(y)\right)^{1/p}.
\end{equation*}
On the other hand, given $\varepsilon>0$, for sufficiently large $r>0$ we have that
\begin{equation*}
\begin{split}
\sum_{y\in(rK+(-1,1)^n)\cap\Z^n}&\bigl(h^{\symbol}\bigr)^{p}(y)\\
=\; &\sum_{y/r\in(K+(1/r)(-1,1)^n)\cap((1/r)\Z^n)}\left(\sup_{u\in(-1,1)^n}f\left(\frac{y+u}{r}\right)\right)^{p}\\
=\; &\sum_{x\in(K+(1/r)(-1,1)^n)\cap((1/r)\Z^n)}\left(\sup_{v\in(1/r)(-1,1)^n}f(x+v)\right)^{p}\\
\leq\; &\sum_{x\in(K+\varepsilon(-1,1)^n)\cap((1/r)\Z^n)}\left(\,\sup_{v\in\varepsilon(-1,1)^n}f(x+v)\right)^{p}\\
\leq\; &\sum_{x\in(K+\varepsilon(-1,1)^n)\cap((1/r)\Z^n)}\bigl(f^{{\symbol}_\varepsilon}\bigr)^{p}(x),
\end{split}
\end{equation*}
where $f^{{\symbol}_\varepsilon}:K+\varepsilon(-1,1)^n\longrightarrow\R_{\geq0}$ is the function given by	
\[f^{{\symbol}_\varepsilon}(z) = \sup_{u\in\varepsilon(-1,1)^n}f(z+u)\]
for all $z\in\R^n$.
Thus, for $r$ large enough, we get
\begin{equation*}
\left(\frac{\binom{n+q}{n}}{\Gsub{n}(rK)}\sum_{y\in(rK)\cap\Z^n}h^{q}(y)\right)^{1/q}
\leq\left(\frac{\binom{n+p}{n}}{\Gsub{n}(rK)}\sum_{x\in(K+\varepsilon(-1,1)^n)\cap((1/r)\Z^n)}
\bigl(f^{{\symbol}_\varepsilon}\bigr)^{p}(x)\right)^{1/p}
\end{equation*}
which implies, using \eqref{e:f_Riem_sums} and \eqref{e:limitrelation_Gn_vol}, that
\begin{equation*}
\left(\frac{\binom{n+q}{n}}{\vol(K)}\int_{K}f^{q}(x)\,\dlat x\right)^{1/q}
\leq\left(\frac{\binom{n+p}{n}}{\vol(K)}
\int_{K+\varepsilon(-1,1)^n}\bigl(f^{{\symbol}_\varepsilon}\bigr)^{p}(x)\,\dlat x\right)^{1/p}.
\end{equation*}
Since $\varepsilon>0$ was arbitrary, to conclude the proof it is enough to show that
\begin{equation}\label{e:final_goal_disc_to_contBerw}
\inf_{\varepsilon>0}\int_{K+\varepsilon(-1,1)^n}\bigl(f^{{\symbol}_\varepsilon}\bigr)^{p}(x)\,\dlat x
\leq \int_{K}f^{p}(x)\,\dlat x.
\end{equation}
To this aim first observe that, by Fubini's theorem, we have (cf.~\eqref{e: sum f(x)_as integral})
\begin{equation}\label{e:Cavalieri_f^symbol}
\int_{K+\varepsilon(-1,1)^n}\bigl(f^{{\symbol}_\varepsilon}\bigr)^{p}(x)\,\dlat x
=\int_{0}^{\infty} p t^{p-1}
\vol\Bigr(\bigl\{x\in K+\varepsilon(-1,1)^n: \, f^{{\symbol}_\varepsilon}(x) > t\bigr\}\Bigr)\,\dlat t.
\end{equation}
Now, since
\[
\bigl\{x\in K+\varepsilon(-1,1)^n: \, f^{{\symbol}_\varepsilon}(x) > t\bigr\}
=\bigl\{x\in K: \, f(x) > t\bigr\} + \varepsilon(-1,1)^n,
\]
we have
\[
\vol\Bigl(\bigl\{x\in K+\varepsilon(-1,1)^n: \, f^{{\symbol}_\varepsilon}(x) > t\bigr\}\Bigr)
\leq \vol\Bigl(\bigl\{x\in K: \, f(x) \geq t\bigr\} + \varepsilon(-1,1)^n\Bigr)
\]
and hence, from \eqref{e:inter_suplevelsets+cubes},
\begin{equation}\label{e:lim_levelsets_f^{symbol}_varepsilon}
\lim_{\varepsilon\to0^+}\vol\Bigl(\bigl\{x\in K+\varepsilon(-1,1)^n: \,
f^{{\symbol}_\varepsilon}(x) > t\bigr\}\Bigr)\\
\leq\vol\Bigl(\bigl\{x\in K: \, f(x)\geq t\bigr\}\Bigr).
\end{equation}
Therefore, taking limits as $\varepsilon\to0^+$ in both sides of \eqref{e:Cavalieri_f^symbol}, applying
the monotone convergence theorem and using \eqref{e:lim_levelsets_f^{symbol}_varepsilon}, we get
\begin{equation*}
\begin{split}
\inf_{\varepsilon>0}\int_{K+\varepsilon(-1,1)^n}\bigl(f^{{\symbol}_\varepsilon}\bigr)^{p}(x)\,\dlat x &=\lim_{\varepsilon\to0^+}\int_{K+\varepsilon(-1,1)^n}\bigl(f^{{\symbol}_\varepsilon}\bigr)^{p}(x)\,\dlat x\\
&\leq \int_{0}^{\infty} p t^{p-1}
\vol\Bigr(\bigl\{x\in K: \, f(x) \geq t\bigr\}\Bigr)\,\dlat t\\
&=\int_{K}f^{p}(x)\,\dlat x.
\end{split}
\end{equation*}
So \eqref{e:final_goal_disc_to_contBerw} follows, which concludes the proof.
\end{proof}


\noindent {\it Acknowledgements.}
We are very grateful to the anonymous
referee for her/his very helpful comments and remarks which have allowed us
to improve the presentation of this work.
We would like to thank M. A. Hern\'andez Cifre for her very
valuable suggestions during the preparation of this paper.

\end{document}